\documentclass[12pt]{amsart}
\usepackage{hyperref} 
\usepackage{amsmath, amsthm, amssymb}
\usepackage{hyperref} 
\usepackage{enumerate}
\usepackage{verbatim}
\usepackage{esint}
\usepackage[T1]{fontenc}

\usepackage{tikz,amsthm,amsmath,amstext,amssymb,amscd,epsfig,euscript, mathrsfs,
 dsfont,pspicture,
multicol,graphpap,graphics,graphicx,times,enumerate,subfig,
sidecap,
wrapfig,color,pict2e}
\usepackage{setspace}

\numberwithin{equation}{section}

\newtheorem{theorem}{Theorem}[section]
\newtheorem{lemma}[theorem]{Lemma}
\newtheorem{corollary}[theorem]{Corollary}

\numberwithin{equation}{section}

\newcommand{\bX}{{\mathbb{X}}}
\newcommand{\bR}{{\mathbb{R}}}
\newcommand{\sF}{{\mathcal{F}}}
\newcommand{\sH}{\mathcal{H}}
\newcommand{\sQ}{\mathcal{Q}}

\newcommand{\bp}{\mathop\mathrm{BP}}
\newcommand{\LG}{\mathop\mathrm{LG}}

\newcommand{\diam}{\mathop\mathrm{diam}}
\newcommand{\rootcube}{\mathop\mathrm{root}}
\newcommand{\dist}{\mathop\mathrm{dist}}

\title{Iterating the Big--Pieces operator and larger sets
}
\author{Jared Krandel}
\address{Department of Mathematics\\ Stony Brook University\\ Stony Brook, NY 11794-3651}
\email{Jared.Krandel@stonybrook.edu}

\author{Raanan Schul}
\address{Department of Mathematics\\ Stony Brook University\\ Stony Brook, NY 11794-3651}
\email{schul@math.sunysb.edu}

\date{July 2021}

\thanks{R.~ Schul was partially supported by the National Science Foundation under Grants No. DMS-1763973.}
\subjclass[2010]{28A75, 30L99.}

\begin{document}

\begin{abstract}
We show that if an Ahlfors-David regular set $E$ of dimension $k$ has Big Pieces of Big Pieces of Lipschitz Graphs (denoted usually by $\bp(\bp(\LG))$), then 
$E\subset \tilde{E}$ where $\tilde{E}$ is Ahlfors-David regular of dimension $k$ and has Big Pieces of Lipschitz Graphs (denoted usually by $\bp(\LG))$.  Our results are quantitative and, in fact, are proven in the setting of a metric space for 
 any family of Ahlfors-David regular sets $\sF$ replacing $\LG$.
A simple corollary is the stability of the BP operator after 2 iterations. 
This was previously only known in the Euclidean setting for the case $\sF= \LG$ with substantially more complicated proofs.
\end{abstract}

\maketitle


\section{Introduction}
A closed set $E$ (with more than one point) in a metric space $\bX$ is said to be $k$-Ahlfors-David regular if there is a constant $C>1$ such that 
for all $r\in (0,\diam(E))$ and $x\in E$ we have
$C^{-1}r^k<\sH^k(E\cap B(x,r))<Cr^k$.
For some given class $\sF$ of $k$-Ahlfors-David regular subsets (of a metric space $\bX$),
we define $\bp(\sF)$ as follows: $F\in \bp(\sF)$ if $F$ is a $k$-Ahlfors-David regular set for which there exists a constant $\theta > 0$ such that for any $x\in F$ and $R>0$, there is a set $G_{x,R}\in \sF$ such that
\begin{equation*}
    \sH^k(B(x,R)\cap F \cap G_{x,R}) \geq \theta \sH^k(B(x,R)\cap F).
\end{equation*}
Conditions involving $\bp(\sF)$ for various classes of sets $\sF$ play an important role in the theory of uniformly rectifiable sets in $\bR^n$ developed By David and Semmes (see e.g. \cite{David-wavelets}, \cite{DS-analysis-of-and-on}).
While the original motivation was the study of  singular integral operators, the study of such conditions has taken on a life of its own. 

In the context of singular integrals, the condition $\bp(\sF)$ is important because it allows the uniform boundedness of a family of SIOs given by convolution with `nice' kernels over sets in $\sF$ to be transported to sets in $\bp(\sF)$. In particular, one can define successively weaker conditions $\bp^j(\sF)$ for all $j > 0$ which all imply boundedness given that the SIOs are bounded on $\sF$; the initial case David and Semmes considered \cite{DS-sing-int-and-rect-sets} used Lipschitz graphs as the base class, i.e., $E\in\bp^j(\LG)$. This raised a natural question: how do the collections 
$\bp^j(\LG)$ behave as $j$ grows?
It turned out that for $j\geq 2$ the collections $\bp^j(\LG)$ are all the same and {their elements} are called {\it Uniformly Rectifiable sets}. We refer the reader to  \cite{DS-analysis-of-and-on}, \cite{DS-sing-int-and-rect-sets}, specifically to Proposition 2.2 on page 97 of  \cite{DS-sing-int-and-rect-sets}, and Theorem 2.29 on page 336 of \cite{DS-sing-int-and-rect-sets}.  
For $n\in[k,2d+1)$ one also needs \cite{JARS-hardsard} to show that BPBI implies BP(BP(LG)), but this is not where most of the work goes --  
the proofs by David and Semmes of that stability (for $j\geq 2$) are quite sophisticated and rely on a Euclidean ambient space.

There has recently been interest in other families $\sF$, in particular for the purpose of studying {\it Parabolic Uniform Rectifiability}.
See e.g. the work in  \cite{bortz2020coronizations}, where the questions about Uniform Rectifiability in the metric setting are discussed for this purpose.
In fact, we refer to \cite{bortz2020coronizations} for a great introduction on contemporary applications of the idea of Big Pieces.

An immediate corollary of the main result contained in this essay (Theorem \ref{t:thm-1})  is that  stabilization of the operator $\bp^j$ occurs in the setting of metric spaces for $j\geq 2$ as well.  Our proof is both simple and direct.

\section{Acknowledgements}
The second author would like to thank Jonas Azzam for pointing out earlier work of David and Semmes.

\section{A Theorem}
\begin{theorem}\label{t:thm-1}
Let $\sF$ be a class of 
(closed)
$k-$Ahlfors-David regular sets in a metric space $\bX$.
Let $E\subseteq \bX$ be a $k$-Ahlfors-David regular set with $E\in\bp(\bp(\sF))$. 
Then there exists a set $F\subset \bX$ such that
\begin{enumerate}[(i)]
    \item {$E \subseteq F$,}
    \item $F$ is $k$-Ahlfors-David regular.
    \item $F\in \bp(\sF)$.
\end{enumerate}
The constants in the conclusion are quantitative with dependance on the constants in the assumptions.
\end{theorem}

\begin{corollary}\label{c:cor-1}
Let $\sF$ be a class of 
closed
$k-$Ahlfors-David regular sets in a metric space $\bX$.
For any $j>2$, and any constants $\theta_1,...,\theta_j>0$ defining $\bp^j(\sF)$, there are $\theta_1', \theta_2'>0$ such that 
the family   $\bp(\bp(\sF))$ defined using $\theta_1', \theta_2'$ is equal to  $\bp^j(\sF)$ defined using  $\theta_1,...,\theta_j$
\end{corollary}
\begin{proof}[Proof of Corollary \ref{c:cor-1}]
Let
$E\in \bp^3(\sF)$.
Then for any $x\in E$ and $R<\diam(E)$ we have a set 
$E'_{x,R}\in \bp^2(\sF)$ such that 
$\sH^k(B(x,R)\cap E)\lesssim \sH^k(B(x,R)\cap E \cap E'_{x,R})$.
By 
Theorem  \ref{t:thm-1}, there is a set 
$F_{x,R}\in \bp(\sF)$ so that $F_{x,R}\supset E'_{x,R}$.
Clearly  
$\sH^k(B(x,R)\cap E)\lesssim \sH^k(B(x,R)\cap E \cap F_{x,R})$.
We have shown $E\in \bp^2(\sF)$.
This gives for any $j\geq 3$ that 
$\bp^j(\sF)=\bp^{j-1}(\sF)$, and so we are done by induction.
\end{proof}

\begin{proof}[Proof of Theorem \ref{t:thm-1} for the case $\diam E < \infty$]
We suppose that  $\diam E < \infty$.
In order to construct the set $F$, we first fix a dyadic cube decomposition of $E$ denoted by $\Delta = \Delta(E)$ with root cube $\rootcube(\Delta) = Q_0 = E$. By construction, for each cube $Q\in\Delta$ there exists a point $c(Q)\in Q$ which we call the center of $Q$ satisfying 
\begin{equation}
    \dist(B(c(Q),c_1 \diam Q),\ E\setminus Q) \geq c_2 \diam Q.
\end{equation}
for some constants $c_1,c_2 > 0$
(see e.g. \cite{Christ-T(b)}).
From now on, define $B_{c(Q)} = B(c(Q), c_1 \diam(Q))$. We construct the set $F$ desired in the theorem inductively. At stage 0, use the fact that $E\in\bp(\bp(\sF))$ to find a closed set $F_{Q_0}\in \bp(\sF)$ such that $F_{Q_0}\subseteq B_{c(Q_0)}$ and
\begin{equation*}
    \sH^k(B_{c(Q_0)} \cap E \cap F_{Q_0}) \gtrsim_{\theta_1} \sH^k(B_{c(Q_0)}\cap E) \gtrsim_{c_1,c_2} \diam(Q_0)^k.
\end{equation*}
We define
\begin{align*}
   F_0 &= F_{Q_0}.
\end{align*} 
We continue the construction by defining a dyadic decomposition $\sQ_1$ of the set $E\setminus F_{Q_0}$. Indeed, since $F_{Q_0}$ is closed, $E\setminus F_{Q_0}$ is relatively open in $E$ and for any $x\in E\setminus F_{Q_0}$, there exists some dyadic cube $Q\ni x$ of maximal diameter such that $\dist(Q, F_{Q_0}) > \diam Q$. We call the disjoint family of all such maximal cubes $\sQ_1$, so that we have
\begin{equation*}
    E\setminus F_{Q_0} = \bigcup_{Q\in\sQ_1} Q.
\end{equation*}
We now give stage 1 of the construction of $F$. For each $Q\in \sQ_1$, Again find closed a set $F_Q\in \bp(\sF)$ such that
\begin{equation}\label{e:2.2}
    \sH^k(B_{c(Q)}\cap E\cap F_Q) \gtrsim_{\theta_1,c_1,c_2} \sH^k(Q).
\end{equation}
We define
\begin{align*}
    F_1 &= F_{Q_0} \cup \bigcup_{Q\in\sQ_1} F_{Q}.
\end{align*}
Continue the construction inductively. Given the construction completed up to stage $m$, we define the set $\sQ_{m+1}$ to be the collection of dyadic cubes with maximal diameter contained in $E \setminus F_m$ such that $Q\in \sQ_{m+1}$ satisfies
\begin{equation}\label{e:2.3}
    \dist(Q, F_m) > \diam(Q). 
\end{equation}
$\sQ_{m+1}$ is a disjoint decomposition of $E\setminus F_m$ so that
\begin{equation}\label{e:2.4}
    E\setminus F_m = \bigcup_{Q\in\sQ_{m+1}} Q.
\end{equation}
Given such a $Q$, let $F_Q\in \bp(\sF)$ with $F_Q\subseteq B_{c(Q)}$ be such that \eqref{e:2.2} holds and define
\begin{equation}\label{e:2.5}
    F_{m+1} = F_m \cup \bigcup_{Q\in\sQ_{m+1}}F_Q = F_{Q_0} \cup \bigcup_{Q\in\sQ_1}F_Q \cup \cdots \cup \bigcup_{Q\in\sQ_{m+1}}F_Q.
\end{equation}
{ Finally, set 
\begin{equation}\label{e:2.6}
    F = \overline{\bigcup_{m=0}^\infty F_m}
\end{equation}}
and define $\sQ = \cup_m \sQ_m$. Now that we have constructed the set $F$, we note two of its simple properties. 
First, given any $Q\neq Q^\prime \in \sQ$, equality \eqref{e:2.3} implies
\begin{equation}\label{e:2.7}
    \dist(F_Q,F_{Q^\prime}) > \min\{\diam(Q),\diam(Q^\prime)\}.
\end{equation}
{ Second,
\begin{equation*}
    \lim F \subseteq E \cup \bigcup_{m=0}^\infty F_m
\end{equation*}
where $\lim F$ denotes the set of limit points of $F$. Indeed, suppose $x\in \lim F$ with $x_j\rightarrow x,\ x_j\in F_{Q_j}$. If the set $\{Q_j\}_j$ is finite, then \eqref{e:2.7} implies the sequence $F_{Q_j}$ is eventually constant, say $F_{Q_j}\rightarrow F_{Q_i}$ meaning $x\in F_{Q_i}$ since $F_{Q_i}$ is closed. If instead $\{Q_j\}_{j}$ is infinite, then consider a subsequence $x_{k_j}\rightarrow x$ such that $Q_{k_j} \not= Q_{k_i}$ for any $i,j$. The fact that $x_{k_j}$ converges combined with \eqref{e:2.7} then implies $\diam Q_j  \rightarrow 0$. Since $\dist(F_{Q_j}, E) \leq \diam Q_j$, we have $\dist(x,E) = 0$ which implies $x\in E$. 
(In particular, we will soon see that this implies $\sH^k(\lim F \setminus \cup_m F_m) = 0$.)
}

{\bf We begin with proving claim (i).} Notice that for any $N\in\mathbb{N}$,
\begin{equation*}
    \sH^k(E\setminus F) \leq \sH^k\left(E\setminus \bigcup_{m=0}^\infty F_m\right) \leq \sH^k(E\setminus F_N)
\end{equation*}
because the sets $F_m$ are increasing. Letting $0 < c_0 < 1$ be the constant implicit in inequality \eqref{e:2.2}, we can write
\begin{align*}
    \sH^k(E\setminus F_N) & \stackrel{\eqref{e:2.5}}{=} \sH^k\left( E\setminus F_{N-1} \setminus \bigcup_{Q\in\sQ_N}F_Q \right) 
       \stackrel{\eqref{e:2.4}}{=} \sH^k\left( \bigcup_{Q\in\sQ_N}Q \setminus \bigcup_{Q\in\sQ_N} F_Q\right)\\
    &= \sum_{Q\in\sQ_N} \sH^k\left( Q\setminus F_Q\right) \leq (1-c_0)\sum_{Q\in\sQ_N}\sH^k(Q) = (1-c_0)\sH^k(E\setminus F_{N-1})
\end{align*}
where we used the fact that $F_Q\cap F_{Q^\prime} = \varnothing$ for $Q,Q^\prime\in\sQ_N$. Since this holds for any $N$, we can iterate this inequality to get
\begin{equation*}
    \sH^k(E\setminus F_N) \leq (1-c_0)^N\sH^k(E)
\end{equation*}
from which we conclude $\sH^k(E\setminus F) = 0$.
To finish the proof of (i), let $x\in E$ be arbitrary. Since $E$ is Ahlfors-David regular, for any $R>0$, $\sH^k(B(x,R)) > 0$ so that $F\cap B(x,r)\not=\varnothing.$ This means $x$ is a limit point of $F$, implying $x\in F$ because $F$ is closed. 

{\bf We now prove (ii).} Fix any point $x\in F$ and some $R < \diam F$. If $x\in F\setminus \cup_m F_m$, then we can find a particular $F_{Q}$ with $\dist(x, F_{Q}) < \frac{R}{100}$ and $\dist(x,F_{Q}) = \dist(x,z)$ for $z\in F_{Q}$. Then, we have $B(z,R/2)\subseteq B(x,R) \subseteq B(z,2R)$, and substitute the first ball or final ball for $B(x,R)$ in the proofs of lower and upper regularity respectively. Hence, we can assume $x\in\cup_m F_m$. By definition, there exists $Q_m\in \sQ_m$ such that $x\in F_{Q_m}$ for some $m\in\mathbb{N}$. Write
\begin{equation}\label{e:2.8}
    \sH^k(B(x,R)\cap F) = \sum_{\substack{ F_Q\cap B(x,R) \not=\varnothing \\ \diam Q > 10R }} \sH^k(B(x,R)\cap F_Q) + \sum_{\substack{ F_Q\cap B(x,R)\not=\varnothing \\ \diam Q \leq 10R }} \sH^k(B(x,R)\cap F_Q).
\end{equation}
We will first show that $F$ is upper regular. Let $\sQ_I$ be the collection of cubes summed over in the first term of \eqref{e:2.8}. 
By \eqref{e:2.7}, we have that for any $Q,Q^\prime\in \sQ_I$, $\dist(F_Q,F_{Q^\prime}) > 10R$. 
This means 
{$\sQ_I$ has at most one element}. 
Given such a $Q$, choose $y\in B(x,R)\cap F_Q$ and write
\begin{equation*}
    \sH^k(B(x,R)\cap F_Q) \leq \sH^k(F_Q\cap B(y,2R)) \lesssim R^k
\end{equation*}
using the fact that $F_Q$ is itself $k$-Ahlfors-David regular. This proves the first sum in \eqref{e:2.8} has the appropriate upper bound. Let $\sQ_{II}$ be the collection of cubes summed over in the second term of \eqref{e:2.8}. Since $\diam(Q) < 10R$, any $Q\in\sQ_{II}$ satisfies $Q\subseteq B(x,20R)$. We first prove a lemma

\begin{lemma}\label{l:lem-1}
Let $Q\in\sQ$, and let $D(Q)$ be the descendants of $Q$ in $\sQ$. Then
\begin{equation*}
    \sH^k\left(\bigcup_{Q^\prime\in D(Q)}F_{Q^\prime}\right) = \sum_{Q^\prime\in D(Q)}\sH^k(F_{Q^\prime}) \lesssim_{\theta_1,c_1,c_2} \sH^k(Q).    
\end{equation*}
\end{lemma}
\begin{proof}[Proof of Lemma \ref{l:lem-1}]
Suppose for simple notation that $Q = Q_0$. Using the regularity of each $F_Q$, we have
\begin{equation*}
    \sH^k\left(\bigcup_{Q\in D(Q_0)}F_{Q}\right) = \sum_{m=0}^\infty\sum_{Q\in D(Q_0)\cap \sQ_m}\sH^k(F_Q) \leq C\sum_{m=0}^\infty\sum_{Q\in D(Q_0)\cap \sQ_m}\sH^k(Q).
\end{equation*}
In analogy to \eqref{e:2.4}, $Q_0\setminus F_{m-1} = \bigcup_{Q\in D(Q_0)\cap \sQ_m}Q$ holds so that
\begin{align*}
    \sum_{Q\in D(Q_0)\cap \sQ_m}\sH^k(Q) &= \sH^k(Q_0\setminus F_{m-1}) = \sH^k\left(Q_0\setminus F_{m-2}\setminus\bigcup_{Q\in D(Q_0)\cap \sQ_{m-1}}F_Q\right) \\
    &= \sH^k\left(\bigcup_{Q\in D(Q_0)\cap \sQ_{m-1}} Q \setminus \bigcup_{Q\in D(Q_0)\cap \sQ_{m-1}}F_Q\right)\\
    &\leq \sum_{Q\in D(Q_0)\cap\sQ_{m-1}}\sH^k(Q\setminus F_Q)\\
    &\leq (1-c_0)\sum_{Q\in D(Q_0)\cap \sQ_{m-1}}\sH^k(Q)
\end{align*}
where $c_0$ was defined as the implicit constant in \eqref{e:2.2}. Iterating this inequality, we find
\begin{equation*}
    \sH^k\left(\bigcup_{Q\in D(Q_0)}F_{Q}\right) \leq C\sum_{m=0}^\infty (1-c_0)^m\sH^k(Q_0) \lesssim_{c_0} \sH^k(Q_0).
\end{equation*}
\end{proof}
Using this lemma, we can write
\begin{align*}
    \sum_{\substack{ F_Q\cap B(x,R)\not=\varnothing \\ \diam Q \leq 10R }} \sH^k(B(x,R)\cap F_Q) &\leq \sum_{\substack{Q\ \text{maximal} \\ Q\in\sQ_{II}}}\sum_{Q^\prime\in D(Q)}\sH^k(F_{Q^\prime}) \lesssim \sum_{\substack{Q\ \text{maximal} \\ Q\in\sQ_{II}}}\sH^k(Q) \\
    &\leq \sH^k(E\cap B(x,20R)) \lesssim R^k.
\end{align*}
This proves the desired bound for the second sum in \eqref{e:2.8}, proving the upper regularity of $F$. Now we show that $F$ is lower regular. If $R < 100\diam Q_m$, then the claim follows immediately from the lower regularity of $F_Q$. 
If $100\diam Q_m \leq R < \diam F$, then 
{since $F_{Q_m}\cap Q_m\not=\varnothing$, there exists $z\in Q$ 
{(and thus, $z\in E$)}
} 
with $B(x,R) \supseteq B(z,R/2)$ and
\begin{equation*}
    \sH^k(B(x,R)\cap F) \geq \sH^k(B(z,R/2)\cap E) \gtrsim R^k
\end{equation*}
using the fact that $E\subseteq F$. This completes the proof of lower regularity, hence of (ii) as well. 

{\bf Finally, we prove (iii).} Fix $x\in F_{Q_m}$ and $R>0$ as in the proof of (ii). Fix a constant $\alpha > 10$ to be chosen later. If $R < \alpha \diam Q_m$, then since $F_{Q_m}\in\bp(\sF)$,  there exists $G_{x,R}\in\sF$ such that
\begin{equation}\label{e:2.9}
    \sH^k(B(x,R)\cap F_{Q_m} \cap G_{x,R}) \geq \theta_2\sH^k(B(x,R)\cap F_{Q_m}) \gtrsim_{C^\prime,\alpha} R^k \gtrsim_{C^{\prime\prime}} \sH^k(B(x,R)\cap F)
\end{equation}
where $C^\prime$ is the regularity constant for $F_{Q_m}$ and $C^{\prime\prime}$ is the regularity constant for $F$. 
Now, suppose that $\alpha \diam Q_m \leq R < \diam F$. Since $x\in F_{Q_m}$, there exists a chain of cubes 
 $Q_i\in\sQ_i,\ 0\leq i \leq m$ such that
\begin{equation*}
    Q_m \subseteq Q_{m-1} \subseteq \ldots \subseteq Q_1 \subseteq Q_0.
\end{equation*}
Next, notice that for any choice of $\alpha > 10$, there exists a smallest cube $Q_j$ in the above chain such that $R < \alpha\diam Q_j$ since for all admissible $R$, $R < 10\diam Q_0$. Choose the constant $\alpha$ such that for any $y\in E\setminus F_i$, the cube $Q_{i+1}\ni y$ satisfies 
\begin{equation}\label{e:2.10}
    \dist(Q_{i+1},F_{Q_i}) < \frac{\alpha}{10}\diam Q_{i+1}.
\end{equation}
In general, $\alpha$ will depend on the constants used in the construction of $\Delta$, as it may be the case that all of the children of the cube $Q_i$ are small relative to $Q_i$ with bounds given in terms of these constants. With such an $\alpha$ chosen, let $Q_j$ be the smallest cube in the above chain for $x$ such that $R < \alpha \diam Q_j$. This means that $R \geq \alpha \diam Q_{j+1}$ so that \eqref{e:2.10} implies that there exists $y\in F_{Q_j}$ such that $B(y, R/2) \subseteq B(x,R)$. We can now repeat the argument of \eqref{e:2.9} with $Q_j$ in place of $Q_m$ to finish the proof.
This completes the proof 
of Theorem \ref{t:thm-1} for the case $\diam E < \infty$.
\end{proof}

Before we turn to the case  $\diam E = \infty$,
we need the following lemma.
It says, roughly, that finite diameter subsets of $E$ can be made regular by extending them slightly. 
This extension also preserves the $\bp(\mathcal{F})$ property.

\begin{lemma}\label{l:lem-2}
Let $E\subseteq\mathbb{X}$ be a $k$-Ahlfors-David regular set and suppose that $G\subseteq E$ satisfies $\diam G = D < \infty$. For any $A \geq 1$, there exists a set $\tilde{G}\subseteq E$ such that
\begin{enumerate}[(i)]
    \item $G\subseteq\tilde{G}\subseteq B(G,\frac{3D}{A})\cap E = \{x\in E:\ d(x,G) < \frac{3D}{A}\}$,
    \item $\tilde{G}$ is $k$-Ahlfors-David regular with constant $C(k, C_E, A)$.
\end{enumerate}
Furthermore, if $E\in\bp(\mathcal{F})$ with constant $\theta_E$ for some class of $k$-Ahlfors-David regular sets, then $\tilde{G}\in\bp(\mathcal{F})$ with constant $\theta(k,\theta_E,A)$.
\end{lemma}

\begin{proof}[Proof of Lemma \ref{l:lem-2}]
    We define an ``interior" of the set $G\subseteq E$ by
    \begin{equation*}
        I_A(G) = \left\{x\in G: d(x,E\setminus G) \geq \frac{D}{A}\right\}.
    \end{equation*}
    The corresponding ``boundary" is then
    \begin{equation*}
        G\setminus I_A(G) = \left\{x\in G: d(x,E\setminus G) < \frac{D}{A}\right\}
    \end{equation*}
    We will construct the set $\tilde{G}$ inductively. In the first stage, we will take a maximal net of appropriate size inside $G\setminus I_A(G)$ and add in balls around each net point to $G$. In the second step, we consider a smaller ``boundary" of this new set and repeat the above process with a finer net and smaller balls. If we continue this process indefinitely while adding balls of exponentially decreasing radii, we get the desired set by taking a closure. We now give this construction explicitly. 
    
    Let $G_0 = G$ and let $X_1$ be a maximal $\frac{D}{A}$-net for the set $G\setminus I_A(G) \subseteq E$. Define
    \begin{equation*}
        G_1 = G \cup \bigcup_{x\in X_1}B\left(x,\frac{2D}{A}\right)\cap E.
    \end{equation*}
    Given the set $G_n$, we define $X_{n+1}$ to be a maximal $4^{-n}\frac{D}{A}$-net for $G_n\setminus I_{4^n\cdot A}(G_n)$ and we let
    \begin{equation*}
        G_{n+1} = G_n \cup \bigcup_{x\in X_{n+1}} B\left(x, 4^{-n}\frac{2D}{A}\right)\cap E.
    \end{equation*}
    Finally, define
    \begin{equation*}
        \tilde{G} = \overline{\bigcup_{n=0}^\infty G_n}.
    \end{equation*}
    We will now show that $\tilde{G}$ satisfies the desired properties in the statement of the lemma. 
    
    { \bf We begin by proving (i)}. The maximal distance of a point $x\in \tilde{G}$ from $G$ is just given by the sum of the radii of the balls added in each step:
    \begin{equation*}
        d(x,G) \leq \frac{2D}{A}\sum_{n=0}^\infty4^{-n} = \frac{8}{3}\frac{D}{A} < \frac{3D}{A}.
    \end{equation*}
    
    { \bf We now prove (ii)}. First, we observe that since $\tilde{G}\subseteq E$, we immediately have, for all $x\in\tilde{G}$, $R>0$,
    \begin{equation*}
        \sH^k(B(x,R)\cap \tilde{G}) \leq \sH^k(B(x,R)\cap E) \leq C_ER^k.
    \end{equation*}
    Hence, $\tilde{G}$ is upper $k$-Ahlfors-David regular with constant $C_E$. We will now show that $\tilde{G}$ is lower regular. In order to do so, we will first prove that there exists a constant $0 < c < 1$ dependent only on $A$ such that
    \begin{equation}\label{e:smallball}
        \forall x\in\tilde{G},\ \forall R,\  0<R<\diam\tilde{G},\ \exists y\in E\ \text{such that}\ B(y,cR)\cap E \subseteq B(x,R)\cap\tilde{G}.    
    \end{equation}
    We note that (ii) will follow from this since for any relevant pair $(x,R)$, we get the existence of $y\in E$ such that
    \begin{equation*}
        \sH^k(B(x,R)\cap\tilde{G}) \geq \sH^k(B(y,cR)\cap E) \geq \frac{c^kR^k}{C_E}
    \end{equation*}
    by the lower regularity of $E$. We now prove \eqref{e:smallball}. We begin by using the constant $c^\prime = \frac{1}{10\cdot 4^4 \cdot A}$ (we will only need to decrease it by a factor of $\frac{1}{2}$ at the end of the proof). Let $x\in\tilde{G}$ and assume $x\in G_m$ for some $m$. There exists some minimal $n$ such that $x\in I_{4^nA}(G_n)$ because $x\in G_m\setminus I_{4^mA}(G_m)$ implies $x\in I_{4^{m+1}A}(G_{m+1})$ by the triangle inequality. Indeed, let $t\in X_{m+1}$ be a nearest net point to $x$ and let $z\in E\setminus G_{m+1}$. We can calculate
    \begin{equation*}
        d(x,z) \geq d(t,z) - d(t,x) \geq 4^{-m}\frac{2D}{A} - 4^{-m}\frac{D}{A} = 4^{-m}\frac{D}{A} > 4^{-m-1}\frac{D}{A}.
    \end{equation*}
    Therefore, $d(x,E\setminus G_{m+1}) > 4^{-m-1}\frac{D}{A}$ so that $x\in I_{4^{m+1}A}(G_{m+1})$. Suppose first that $n \leq 4$. In this case, we will take $y = x$, and we must show the inclusion of the balls given in \eqref{e:smallball} for any admissible value of $R$. For $0 < R \leq 4^{-4}\frac{D}{A}$, note that $x\in I_{4^4A}(G_4)$ implies 
    \begin{equation}\label{e:basecase}
        d(x,E\setminus\tilde{G}) \geq d(x,E\setminus G_4) > \frac{D}{4^4A}
    \end{equation}
    so that $B(x,R)\cap\tilde{G} = B(x,R)\cap E$. If instead $4^{-4}\frac{D}{A} < R < \diam\tilde{G} < D + \frac{6D}{A} < 10D$, 
    \begin{equation*}
        c^\prime R = \frac{R}{10\cdot4^4\cdot A} < \frac{10D}{10\cdot4^4\cdot A} = 4^{-4}\frac{D}{A}.
    \end{equation*}
    Which shows that 
    \begin{equation*}
         B(x,c^\prime R)\cap E\subseteq B\left(x,4^{-4}\frac{D}{A}\right)\cap E = B\left(x,4^{-4}\frac{D}{A}\right)\cap\tilde{G}
    \end{equation*}
    by \eqref{e:basecase}.
    Now, suppose $n > 4$. This means $x\in I_{4^nA}(G_n)\setminus I_{4^{n-1}A}(G_{n-1})$.
    Hence, if $R < 4^{-n}\frac{D}{A}$, then we can take $y = x$ and note that $B(x,R)\cap\tilde{G} = B(x,R)\cap E$ in analogy to \eqref{e:basecase}. Now, suppose $4^{-m}\frac{D}{A} \leq R < 4^{-m+1}\frac{D}{A}$ for $0 \leq m \leq n-3$. There exist net points $x_p\in X_p$ for $m+3 \leq p \leq n$ such that
    \begin{align*}
        d(x,x_n) &\leq 4^{-n}\frac{2D}{A},\\
        d(x_{p+1},x_{p}) &\leq 4^{-p}\frac{2D}{A}.
    \end{align*}
    Hence, the triangle inequality implies
    \begin{equation}\label{e:dist}
        d(x,x_{m+3}) \leq \frac{2D}{A}\sum_{p=m+2}^n4^{-p} \leq \frac{2D}{A}\left(4^{-m-2}\cdot 2\right)= 4^{-m-1}\frac{D}{A}.
    \end{equation}
    In this case, we choose $y=x_{m+3}$. We calculate
    \begin{equation*}
        B(y,c^\prime R) = B\left(x_{m+3}, \frac{R}{10\cdot4^4\cdot A}\right) \subseteq B\left(x_{m+3}, 4^{-(m+3)}\frac{D}{10A^2}\right)\subseteq B\left(x,4^{-m}\frac{D}{A}  \right)\subseteq B(x,R)
    \end{equation*}
    using \eqref{e:dist} and the fact that $4^{-m}\frac{D}{A} \leq R < 4^{-m+1}\frac{D}{A}$. In the case when $\frac{D}{A} < R < 10D$, choose $y = x_3$, the nearest net point in $X_3$ and observe that 
    \begin{equation*}
         B(y,c^\prime R) = B\left(x_{3}, \frac{R}{10\cdot4^4\cdot A}\right) \subseteq B\left(x_{m+3}, 4^{-4}\frac{D}{A}\right)\subseteq B\left(x,\frac{D}{A} \right)\subseteq B(x,R)
    \end{equation*}
    again using \eqref{e:dist}. This proves \eqref{e:smallball} for all $x\in G_n$ for some $n$. If $x\not\in G_n$ for all $n$, then given any admissible $R>0$, there is a net point $t\in X_N$ for arbitrarily large $N$ such that $d(x,t) < \frac{R}{4}$ so that $B(t,\frac{R}{2}) \subseteq B(x,R)$ and, applying \eqref{e:smallball} to $B(t,\frac{R}{2})$, we get a point $y\in B(t,\frac{R}{2})$ such that $B(y,c^\prime \frac{R}{2})\subseteq B(t,\frac{R}{2})\subseteq B(x,R)$. Take $c = \frac{c^\prime}{2}$ and $B(y,cR)\subseteq B(x,R)$ so that $\eqref{e:smallball}$ holds with $c = \frac{1}{20\cdot4^4\cdot A}$.
    
    {\bf Proof that $\tilde{G}\in\bp(\mathcal{F})$.} This follows from \eqref{e:smallball}. Indeed, for any admissible pair $(x,R)$, choose $y$ as given by \eqref{e:smallball}. Applying the $\bp(\mathcal{F})$ condition for $E$ in the ball $B(y,cR)$ gives a set $H_{y,cR}\in\mathcal{F}$ such that
    \begin{equation*}
        \sH^k(B(x,R)\cap\tilde{G}\cap H_{y,cR}) \geq \sH^k(B(y,cR)\cap E \cap H_{y,cR}) \gtrsim_{A, \theta_E, k} R^k \gtrsim_C \sH^k(B(x,R)\cap\tilde{G}).
    \end{equation*}
This concludes the proof of the lemma.
\end{proof}

\begin{proof}[Proof of Theorem \ref{t:thm-1} for the case $\diam E = \infty$]
Fix $x_0\in E$. Let $A > 1$ and, for $n\geq0$, set
\begin{equation*}
    B_n = B(x_0, A^n)
\end{equation*}
where the constant $A$ is sufficiently large in terms of $C_E,k,$ and $\theta_E$, the $\bp$ constant. Let $E_n$ be the Ahlfors-David regular extension of the set $E\cap B_n$ with constant $A$ in Lemma \ref{l:lem-2} replaced with 100 so that $E_n\subseteq B(E\cap B_n, \frac{A^{n}}{4})$. $E_n$ satisfies the hypotheses of the finite diameter case of the theorem, so apply the theorem to get a regular set $F_n\in\bp(\mathcal{F})$ satisfying
\begin{equation*}
    E_n\subseteq F_n \subseteq B\left(x_0, \frac{5A^n}{4}\right).
\end{equation*}
In order to ensure bounded overlap, we then define $\tilde{F}_0 = F_0$ and $\tilde{F}_n$ for $n\geq 1$ to be the regular extension of $F_n\setminus\frac{1}{2}B_{n-1}$ given by the lemma with constant $A$ there replaced by $100A$ here. By construction, $\tilde{F}_n \subseteq B(F_n, \frac{A^{n-1}}{10})$ so that $\tilde{F}_n\cap\frac{1}{4}B_{n-1} = \varnothing$ and $\tilde{F}_n\subseteq B(x_0,2A^n)$. We also have $\tilde{F}_n\in\bp(\mathcal{F})$ with constant $\tilde{\theta}_F$ independent of $n$. We now define
\begin{equation*}
    F = \bigcup_{n=0}^\infty\tilde{F}_n
\end{equation*}
and claim that $F$ satisfies conditions (i)-(iii).

{\bf Proof of (i).} By definition, $E\cap (B_n\setminus \frac{1}{2}B_{n-1})\subseteq \tilde{F}_n$ so $E = \bigcup_{n=0}^\infty E\cap(B_n\setminus \frac{1}{2}B_{n-1}) \subseteq F$.

{\bf Proof of (ii).} For any $n$, $\tilde{F}_n$ is regular with some constant $\tilde{C}_F(A,C_E,k)$ independent of $n$. Lower regularity of $F$ with constant $\tilde{C}_F$ follows immediately, so we only need to show that $F$ is upper regular. Let $x\in\tilde{F}_n$ for some $n$. Observe that, for $j \geq 2$
\begin{equation*}
    d(x,\tilde{F}_{n+j}) \geq d\left(\tilde{F}_n,\frac{1}{4}B_{n+j-1}\right) \geq \frac{1}{4}A^{n+j-1} - 2A^n > A^{n+j-2}
\end{equation*}
provided we choose $A$ sufficiently large. Hence, if $R \leq A^{n-2}$, then $B(x,R)\cap \tilde{F}_j=\varnothing$ for $|n-j| \geq 2$. In this case,
\begin{align*}
    \sH^k(B(x,R)\cap F) = \sum_{j=-1}^1\sH^k(B(x,R)\cap\tilde{F}_{n+j}) \lesssim_{\tilde{C}_F} R^k
\end{align*}
independent of $n$ because $\tilde{F}_{n+j}$ is regular with constant independent of $n$. Now, suppose $A^{j} < R \leq A^{j+1}$ for $j \geq n-2$. We can write
\begin{align*}
    \sH^k(B(x,R)\cap F) &= \sum_{i=0}^{j+2}\sH^k(B(x,R)\cap\tilde{F}_i) \leq \sum_{i=0}^{j+2}\sH^k(\tilde{F}_i) \leq \tilde{C}_F\sum_{i=0}^{j+2}\diam(\tilde{F}_i)^k\\ 
    & \leq \tilde{C}_F\sum_{i=0}^{j+2} (4A)^{ik} \leq 2\tilde{C}_F(4A)^{(j+2)k} \leq (4A)^{2k+1}\tilde{C}_F(4R)^k.
\end{align*}
This proves upper regularity and finishes the proof of (ii). From now on, let $C_F = C_F(C_E,A,k)$ be the regularity constant for $F$.

{\bf Proof of (iii)} Let $x\in\tilde{F}_n$ and $R>0$. Suppose first that $0 < R \leq A^{n+2}$. Because $\tilde{F}_n\in\bp(\mathcal{F})$ by the lemma with constant $\tilde{\theta}_F(\theta_E,A,k)$ independent of $n$, we get the existence of a set $G_{x,R}\in\mathcal{F}$ such that 
\begin{align*}
    \sH^k(B(x,R)\cap F \cap G_{x,R}) &\geq \sH^k(B(x,R) \cap \tilde{F}_n\cap G_{x,R}) \gtrsim_{\tilde{\theta}_F, A} \sH^k(B(x,R)\cap \tilde{F}_n) \\
    &\gtrsim_{C_F} R^k \gtrsim_{C_F} \sH^k(B(x,R)\cap F).
\end{align*}
using the fact that $\tilde{F}_n$ is regular. Now, suppose $A^j < R \leq A^{j+1}$ for $j \geq n+2$. Because $x\in \tilde{F}_n,\ \frac{1}{4}A^{n-1} \leq d(x, x_0) \leq 2A^n$ so that
\begin{equation*}
    \tilde{F}_{j-2} \subseteq B\left(x_0, 2A^{j-2}\right) \subseteq B\left(x, 2A^{j-2} + 2A^n\right) \subseteq B(x, A^{j-1}) \subseteq B(x,R).
\end{equation*}
Using the above containment and the fact that $\tilde{F}_{j-2}\in\bp(\mathcal{F})$, there exists a set $G_{x,R}\in\mathcal{F}$ with both $G_{x,R}\subseteq B(x,R)$ and 
\begin{equation*}
    \sH^k(G_{x,R}\cap \tilde{F}_{j-2}) \gtrsim_{\tilde{\theta}_F} \diam(\tilde{F}_{j-2})^k \gtrsim A^{(j-2)k} \gtrsim_{A,k} R^k.
\end{equation*}
Hence, we have $\sH^k(B(x,R)\cap G_{x,R} \cap F) \gtrsim_{\tilde{\theta}_F,A,k} R^k \gtrsim_{C_F} \sH^k(B(x,R)\cap F)$ as desired.
\end{proof}

\bibliography{bib-file-2}{}
\bibliographystyle{plain}

\end{document}